\documentclass[11pt]{amsart}
\usepackage{latexsym}
\usepackage{fullpage}
\usepackage{amssymb}
\usepackage{amscd}
\usepackage{float}
\usepackage{graphicx}
\usepackage{amsfonts}
\usepackage{pb-diagram}
\usepackage{amsmath,amscd}
\usepackage{caption}
\usepackage{subcaption}
\usepackage{xcolor}

\newtheorem{thm}{Theorem}

\newtheorem{prop}{Proposition}
\newtheorem{defn}{Definition}
\newtheorem{remark}{Remark}

\begin{document}

\title[Tied pseudo links \& Pseudo knotoids]
  {Tied pseudo links \& Pseudo knotoids}

\author{Ioannis Diamantis}
\address{China Agricultural University,
International College Beijing, No.17 Qinghua East Road, Haidian District,
Beijing, {100083}, P. R. China.}
\email{ioannis.diamantis@hotmail.com}



\setcounter{section}{-1}

\date{}

\begin{abstract}
In this paper we study the theory of {\it pseudo knots}, which are knots with some missing crossing information, and we introduce and study the theory of {\it pseudo tied links} and the theory of {\it pseudo knotoids}. In particular, we first present a braiding algorithm for pseudo knots and we then introduce the $L$-moves in that setting, with the use of which we formulate a sharpened version of the analogue of the Markov theorem for pseudo braids. Then we introduce and study the theory of {\it tied pseudo links}, that generalize the notion of tied links, and we exploit the relation between tied pseudo links and tied singular links. We first present an $L$-move braid equivalence for tied singular braids. Then, we introduce the tied pseudo braid monoid and we formulate and prove analogues of the Alexander and Markov theorems for tied pseudo links. Finally, we introduce and study the theory of {\it pseudo knotoids}, that generalize the notion of knotoids. We present an isotopy theorem for pseudo knotoids and we then pass to the level of braidoids. We further introduce and study the {\it pseudo braidoids} by introducing the {\it pseudo} $L${\it -moves} and by presenting the analogues of the Alexander and Markov theorems for pseudo knotoids. We also discuss further research related to {\it tied (multi)-knotoids and tied pseudo (multi)-knotoids}. The theory of pseudo knots may serve as a strong tool in the study of DNA, while tied links have potential use in other aspects of molecular biology.
\end{abstract}

\maketitle

\section{Introduction}\label{intro}

Pseudo diagrams of knots, links and spatial graphs were introduce by Hanaki in \cite{H} as projections on the 2-sphere with over/under information at some of the double points. They comprise a relatively new and important model for DNA knots, since there exist cases of DNA knots that, after studying them by electron microscopes, it is hard to say a positive from a negative crossing. By considering equivalence classes of pseudo diagrams under equivalence relations generated by a specific set of Reidemeister moves (see Figure~\ref{reid}), one obtains the theory of {\it pseudo knots}, that is, standard knots whose projections contain crossings with missing information (see \cite{HJMR}). In \cite{BJW}, the pseudo braid monoid is introduced, which is related to the singular braid monoid, and with the use of which, the authors present the analogues of the Alexander and Markov theorems for pseudo knots (see Theorems~\ref{alexpl} \& \ref{markpl} in this paper). In this paper we introduce the $L$-moves for pseudo braids and we obtain a sharpened version of the analogue of the Markov theorem for pseudo braids. Moreover, we introduce and study the notion of {\it tied pseudo links} and {\it tied pseudo braids}, that is, pseudo links/braids equipped with {\it ties}, which are non-embedded arcs joining some components of the pseudo link. 

\smallbreak

Tied links were introduced in \cite{AJ1} as a generalization of links in $S^3$ forming a new class of knotted objects. They are classical links equipped with {\it ties} and they are obtained by considering the closure of tied braids, which appear from a diagrammatic interpretation of the defining generators of the {\it algebra of braids and ties} introduced and studied in \cite{Ju, AJ2}. The diagrammatic theory of tied links have potential usefulness in other aspects of molecular biology. Of particular interest are the {\it tied singular links}, introduced and studied in \cite{AJ3}, that is, singular links equipped with ties. They constitute a generalization of singular links in $S^3$ in the same way that tied links generalize standard links. For that reason, we present an $L$-move braid equivalence for tied singular braids (see Theorem~\ref{lmarktsl}) and we show how tied singular links are related to tied pseudo links. We then define the {\it tied pseudo braid monoid}, which is related to the {\it tied singular braid monoid} (\cite{AJ3}) and which is fundamental for this paper. The tied pseudo braid monoid is the basis toward formulating the analogues of the Alexander and Markov theorems for tied pseudo links, that we then state and prove (see Theorems~\ref{alexthmtpl}, \ref{marktpb} \& \ref{marktpb2}).
 
\smallbreak

We finally extend the notion of knotoids, introduced in \cite{T} as open knotted curves in oriented surfaces, to that of {\it pseudo knotoids}, i.e. knotoids with some missing crossing information. We present an isotopy theorem for pseudo knotoids and we then study the theory of {\it pseudo braidoids}. Braidoids were introduced in \cite{GL2} as the counterpart analogue of knotoids, and the authors present a braiding algorithm, as well as braidoid equivalence moves (see also \cite{GL1}). In this paper we present the analogue of the Alexander theorem for pseudo knotoids (see Theorem~\ref{alexpknoid}) and we then introduce $L$-moves on pseudo braidoids, with the use of which, we state and prove the analogue of the Markov theorem for pseudo braidoids (see Theorem~\ref{breqpbroid}). Furthermore, we point out the diagrammatic relation between pseudo knotoids and the recent theory of singular knotoids (see \cite{MLK}). Finally, it is worth mentioning that the theory of knotoids has been used to classify entanglement in proteins, which are long chains of amino acids that sometimes form open ended knots \cite{DGBS, GDBS, GGLDSK}.

\smallbreak

The paper is organized as follows: \S~\ref{sectl} we recall results concerning tied links and tied braids from \cite{AJ1}, in \S~\ref{pknot} we recall results on pseudo knots and singular knots, that play an important role for the rest of the paper, and in \S~\ref{knbroid} we recall all necessary results for knotoids and braidoids from \cite{T} and \cite{GL1, GL2}. In \S~\ref{lpb} we introduce $L$-moves for pseudo braids and in \S~\ref{alse} we present a braiding algorithm for pseudo knots. In \S~\ref{lmpb} we prove a sharpened version of the analogue of the Markov theorem for pseudo braids in $S^3$. We then proceed by introducing the family of tied pseudo links (see \S~\ref{tpl}) and the tied pseudo braid monoid. In \S~\ref{tpbal} we formulate and prove the analogue of the Alexander theorem for tied pseudo links and in \S~\ref{mthmtpb} we present the analogue of the Markov theorem for tied pseudo braids. The last part of this paper concerns the theory of pseudo knotoids, introduced in in \S~\ref{pkoidnew}. In \S~\ref{pbrd} we introduce the notion of pseudo braidoids and we adopt the braiding algorithm of \S~\ref{tpbal} for the case of pseudo knotoids. We also introduce the notion of pseudo $L$-moves, with the use of which, we formulate the analogue of the Markov theorem for pseudo knotoids. Finally, in \S~\ref{fres} we introduce the notions of tied knotoids and tied pseudo knotoids, which are the topics of further research.

\section{Preliminaries}\label{prel}

\subsection{Tied links \& tied braids}\label{sectl}

In this subsection we introduce tied links and we review the tied braid monoid introduced in \cite{AJ1}. Roughly speaking, a tied link is a link whose set of components is subdivided into classes. More precisely:

\begin{defn}\label{tls3}\rm
A {\it (oriented) tied link} $L(P)$ on $n$ components is a set $L$ of $n$ disjoint smooth (oriented) closed curves embedded in $S^3$, and a set $P$ of ties, i.e., unordered pairs of points $(p_r, p_s)$ of such curves between which there is an arc called a {\it tie}. If $P=\emptyset$, then $L$ is a classical link in $S^3$. See Figure~\ref{tkn1} for an example.
\end{defn}

Note that ties are depicted as red \color{red} springs \color{black}or \color{red}line segments \color{black}connecting pairs of points lying on the curves and that ties are {\bf not} embedded arcs, i.e. arcs can cross through ties. The set of ties on a tied link defines a partition of the set of components of the link by considering components of the tied link to belong to the same class, if there is a tie connecting them. Moreover, a {\it tied link diagram} is defined as a diagram of a tied link.

\begin{figure}[ht]
\begin{center}
\includegraphics[width=2.4in]{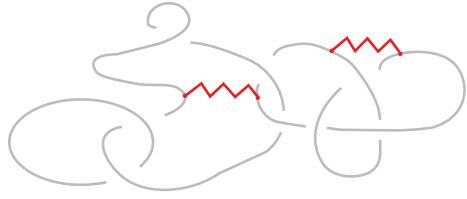}
\end{center}
\caption{A tied link.}
\label{tkn1}
\end{figure}

{\it Tie isotopy}, i.e. isotopy between tied links, is defined as ambient isotopy between links (ignoring the ties), taking also into consideration that the set of ties in the links define the same partition of the set of components. More precisely:

\begin{defn}\label{istls3}\rm
Two (oriented) tied links $L(P)$ and $L^{\prime}(P^{\prime})$ are {\it tie isotopic} if:
\smallbreak
\begin{itemize}
\item the links $L$ and $L^{\prime}$ in $S^3$ are ambient isotopic and
\smallbreak
\item the sets $P$ and $P^{\prime}$ define the same partition of the set of components of $L$ and $L^{\prime}$.
\end{itemize}
\end{defn}

The partition defined on the link components by the set of ties induces a tie-equivalence relation, which implies that a tie may be introduced or deleted at will between any two points of the same component. Also, an extra tie may be introduced between two components that are already connected by a tie. Finally, a tie may be introduced between two components which are both tied with a third one.

\smallbreak

{\it Tied braids}, that is, standard braids in $S^3$ equipped with ties and whose standard closure gives rise to tied links, were also introduced in \cite{AJ1}. They generalize the notion of standard braids and via the {\it monoid of tied braids}, i.e. the analogue of the braid group for tied links, one may extend the Alexander and Markov theorems for oriented tied links. In particular:

\begin{defn}\label{montls3}\rm
The tied braid monoid $TM_{n}$ is the monoid generated by $\sigma_1, \ldots, \sigma_{n-1}$, the usual braid generators of $B_{n}$, and the generators $\eta_1, \ldots, \eta_{n-1}$, called ties (see Figure~\ref{gs3}), satisfying the braid relations of $B_{n}$ together with the following relations:

\[
\begin{array}{lcl}
\eta_i \eta_j & = & \eta_j \eta_i, \qquad {\rm for\ all}\ i, j\\
&&\\
\eta_i \sigma_i & = & \sigma_i \eta_i, \qquad {\rm for\ all}\ i\\
&&\\
\eta_i \sigma_j & = & \sigma_j \eta_i, \qquad {\rm for\ all}\ |i-j|>1\\
&&\\
\eta_i \sigma_j \sigma_i^{\epsilon} & = & \sigma_j \sigma_i^{\epsilon} \eta_j, \quad {\rm for}\ |i-j|=1\ {\rm and\ where}\ \epsilon\, =\, \pm 1\\
&&\\
\eta_i \eta_j \sigma_i & = & \eta_j \sigma_i \eta_j\ =\ \sigma_i \eta_i \eta_j, \quad {\rm for}\ |i-j|=1\\
&&\\
\eta_i^2 & = & \eta_i, \qquad \ \ {\rm for\ all}\ i\\
\end{array}
\]
\end{defn}

\begin{figure}[ht]
\begin{center}
\includegraphics[width=3.2in]{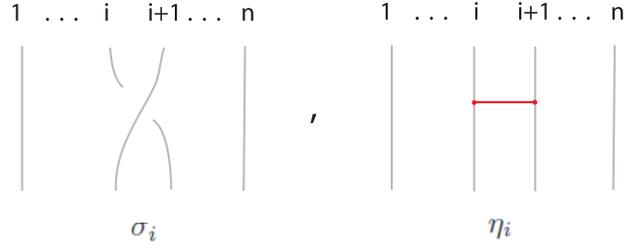}
\end{center}
\caption{Generators of $TM_n$.}
\label{gs3}
\end{figure}

\begin{figure}[ht]
\begin{center}
\includegraphics[width=3.7in]{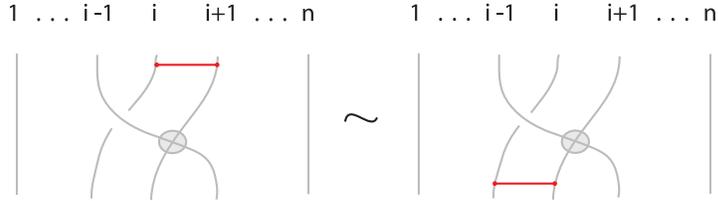}
\end{center}
\caption{Relations in $TM_{n}$, where the shaded crossing is either positive or negative.}
\label{relt1}
\end{figure}

\begin{figure}[ht]
\begin{center}
\includegraphics[width=3.3in]{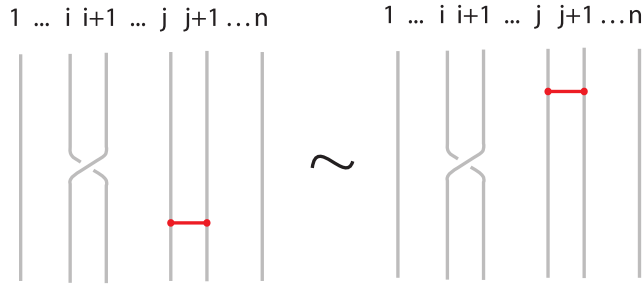}
\end{center}
\caption{The relations $\sigma_i\, \eta_{j}\, =\, \eta_{j}\, \sigma_i$.}
\label{relt2}
\end{figure}

\begin{figure}[ht]
\begin{center}
\includegraphics[width=3.3in]{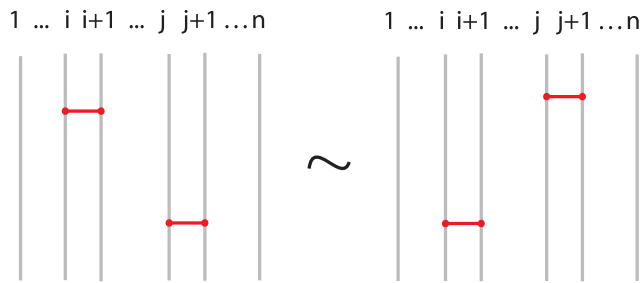}
\end{center}
\caption{The relations $\eta_i\, \eta_{j}\, =\, \eta_{j}\, \eta_i$.}
\label{relt3}
\end{figure}

\noindent Note that in \cite{AJ1} the tied braid monoid is denoted as $TB_n$. Moreover, in \cite{AJ1} the authors also introduce the {\it generalized ties} $\eta_{i, j}$ and present their properties, with the use of which, one obtains the {\it mobility property} for tied links (see Proposition~\ref{propaj}i. below), a useful property that allows us to prove the analogue of the Markov theorem for tied links. More precisely:

\begin{defn}\rm
The {\it generalized tie} $\eta_{i, j}$, joining the $i^{th}$ with the $j^{th}$ strand is defined as:

\begin{equation}\label{genties}
\eta_{i, j}\ :=\ \sigma_i\, \cdots\, \sigma_{j-2}\, \eta_{j-1}\, \sigma_{j-2}^{-1}\, \cdots\, \sigma_i^{-1}
\end{equation}

Clearly, $\eta_{i, i+1}=\eta_i$. Set now $\eta_{i, i}:=1$ and define also the {\it length} of a generalized tie $\eta_{i, j}$ as $l(\eta_{i, j})=|i-j|$, where $l(\eta_i)=1$.
\end{defn}

\smallbreak

Generalized ties $\eta_{i, j}$ are {\it transparent} with respect to all strands between the $i^{th}$ and $j^{th}$ strands, that is, they can be drawn no matter if in front or behind these strands. Moreover, the generalized ties are provided with {\it elasticity}, that is, each generalized tie can be transformed by a second Reidemeister move in which the tie is stretched and represented as a spring. Finally, the generalized ties satisfy the following relations (\cite{AJ1}):

\begin{equation}\label{eqf}
\begin{array}{crclcrcl}
(i) & \sigma_i\, \eta_{i, j} & = & \eta_{i+1}\, \sigma_i, & (ii) & \sigma_j\, \eta_{i, j} & = & \eta_{i, j+1}\, \sigma_i,\\
&&&&&&&\\
(iii) & \sigma_{i-1}\, \eta_{i, j} & = & \eta_{i-1, j}\, \sigma_{i-1}, & (iv) & \sigma_{j-1}\, \eta_{i, j} & = & \eta_{i, j-1}\, \sigma_{j-1}, \\
&&&&&&&\\
(v) & \eta_{i, k}\, \eta_{k, m} & = & \eta_{i, k}\, \eta_{i, m} & = & \eta_{k, m}\, \eta_{i, m},& & 1\leq i, k, m\leq n\\
\end{array}
\end{equation}

\noindent The relations in Definition~\ref{montls3}, Equations~\ref{eqf} and the fact that the generalized ties are transparent and equipped with elasticity, imply that in a tied braid all ties may slide to the bottom part of the braid (\cite{AJ1}). In particular:

\begin{prop} \label{propaj}
\begin{itemize}
\item[i.] Mobility property: In any tied braid all ties can be moved to the bottom (or to the top), that is, $\alpha\, \sim\, \beta\, \gamma$, where $\alpha\in TM_n, \beta \in B_n$ and $\gamma$ the set of generalized ties.
\smallbreak
\item[ii.] Any set of generalized ties in $TM_n$ defines an equivalence relation on the set of $n$ strands.
\smallbreak
\item[iii.] Let $\alpha_1=\beta_1\, \gamma_1,\ \alpha_2=\beta_2\, \gamma_2$ be two tied braids in $TM_n$. Then, $\alpha_1\, \sim\, \alpha_2$, if and only if $\beta_1\, \sim\, \beta_2$ in $B_n$ and $\gamma_1, \gamma_2$ define the same partition of the set of the strands. 
\end{itemize}
\end{prop}

By considering now the obvious inclusion $TM_n \subset TM_{n+1}$, we can consider the inductive limit $TM_{\infty}$. Using the generalized ties and the fact that ties can be transferred freely, in \cite{AJ1} it is shown that the braiding algorithm in \cite{LR1} can also be applied for tied links. In particular:

\begin{thm}[{\bf The analogue of the Alexander theorem for tied links in} $S^3$] \label{alextls3}
Every oriented tied link is isotopic to the closure of a tied braid.
\end{thm}

In \cite{AJ1} it is also proved the analogue of the Markov theorem for tied braids using the mobility property and Equations~\ref{eqf}. More precisely, we have the following theorem:

\begin{thm}[{\bf The analogue of the Markov Theorem for tied braids}] \label{marktls3}
Two tied braids have tie isotopic closures if and only if one can  obtained from the other by a finite sequence of the following moves:

\[
\begin{array}{llcll}
{Conjugation:} &  \alpha\, \beta & \sim & \beta\, \alpha, & {\rm for\ all}\ \alpha,\, \beta \in TM_n,\\
&&&&\\
{Stabilization:} &  \alpha & \sim & \alpha\, \sigma_n^{\pm 1}, & {\rm for\ all}\ \alpha \in TM_n,\\
&&&&\\
{Ties:} &  \alpha & \sim & \alpha\, \eta_{i, j}, & {\rm for\ all}\ \alpha \in TM_n\ {\rm such\ that}\\
\end{array}
\]
\smallbreak
\noindent $s_{\alpha}(i)=j$, {\rm where} $s_{\alpha}$ {\rm denotes the permutation associated to the braid obtained from} $\alpha$ {\rm by forgetting its ties.}
\end{thm}

\begin{remark}\rm
It is also worth mentioning that in \cite{AJ2} the authors prove that the bt-algebra (a quotient of the tied braid monoid) supports a Markov trace. Moreover, in \cite{AJ3}, the authors provide a purely algebraic and combinatorial version of tied links, with the use of which, they prove first that the tied braid monoid has a decomposition like a semi-direct product and present new proofs (different from the proofs in \cite{AJ1}) for the analogues of the Alexander and Markov theorems for tied links.
\end{remark}

\subsection{Pseudo knots \& singular knots}\label{pknot}

In this subsection we review the theory of pseudo knots and pseudo braids, introduced in \cite{HJMR, BJW}. A {\it pseudo knot diagram} consists of a regular knot diagram where some crossing information may be missing, that is, it is unknown which strand passes over and which strand passes under the other. These undetermined crossings are called {\it pre-crossings} (for an illustration see Figure~\ref{pk1}).

\begin{figure}[ht]
\begin{center}
\includegraphics[width=1.8in]{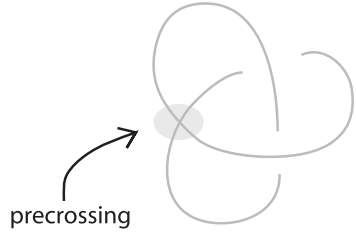}
\end{center}
\caption{A pseudo knot.}
\label{pk1}
\end{figure}

\begin{defn}\rm
A {\it pseudo knot} is defined as equivalence class of pseudo knot diagrams under all versions of the classical Reidemeister moves and the extended pseudo-Reidemeister moves illustrated in Figure~\ref{reid}.
\end{defn}

\begin{figure}[ht]
\begin{center}
\includegraphics[width=6.2in]{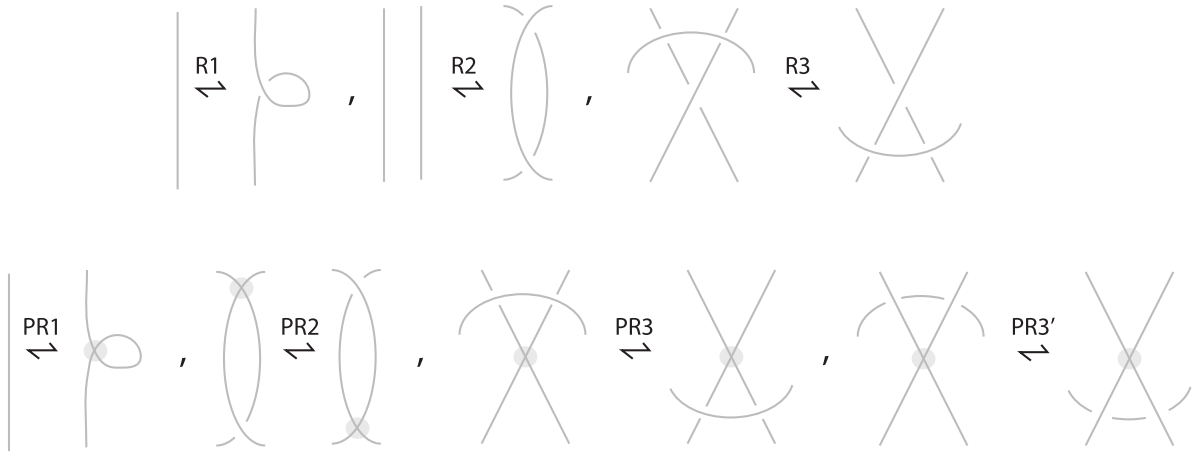}
\end{center}
\caption{Reidemeister moves for pseudo knots.}
\label{reid}
\end{figure}

As explained in \cite{BJW}, pseudo knots are closely related to {\it singular knots}, that is, knots that contain a finite number of (rigid) self-intersections. In particular, there exists a bijection $f$ from the set of singular knot diagrams to the set of pseudo knot diagrams where singular crossings are mapped to pre-crossings. In that way we may also recover all of the pseudo-Reidemeister moves, with the exception of the pseudo-Reidemeister I move (PR1). Moreover, $f$ induces an onto map from the set of singular knots to the set of pseudo knots, since the images of two isotopic singular knot diagrams are also isotopic pseudo knot diagrams with exactly the same sequence of corresponding Reidemeister moves.

\smallbreak

We now introduce the pseudo braid monoid, $PM_n$, following \cite{BJW}.

\begin{defn}\label{pmn}\rm
The {\it pseudo braid monoid}, $PM_n$, is the monoid generated by $\sigma_i^{\pm 1}, p_i, i=1, \ldots, n-1$, illustrated in Figure~\ref{gens}, where the $\sigma_i$'s generate the classical braid group $B_n$ and the $p_i$'s are the elementary pre-crossings between the $i^{th}$ and the $(i+1)^{st}$ strand, satisfying the following relations:
\[
\begin{array}{rlcll}
{\rm i.} & p_i\, p_j & = & p_j\, p_i, & {\rm if}\ |i-j|\geq 2\\
&&&&\\
{\rm ii.} & p_i\, \sigma_j^{\pm 1} & = & \sigma_j^{\pm 1}\, p_i, & {\rm if}\ |i-j|\geq 2\\
&&&&\\
{\rm iii.} & p_i\, \sigma_i^{\pm 1} & = & \sigma_i^{\pm 1}\, p_i, & i=1, \ldots, n-1\\
&&&&\\
{\rm iv.} & \sigma_i\, \sigma_{i+1}\, p_i & = & p_{i+1}\, \sigma_i\, \sigma_{i+1}, & i=1, \ldots, n-2\\
&&&&\\
{\rm v.} & \sigma_{i+1}\, \sigma_i\, p_{i+1} & = & p_{i}\, \sigma_{i+1}\, \sigma_i, & i=1, \ldots, n-2\\
\end{array}
\]
\end{defn}

\begin{figure}[ht]
\begin{center}
\includegraphics[width=2.8in]{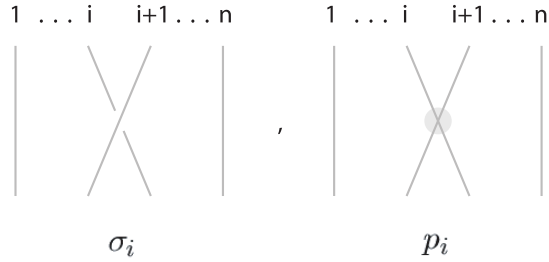}
\end{center}
\caption{The pseudo braid monoid generators.}
\label{gens}
\end{figure}

We denote a singular crossing by $\tau_i$ and we have that, if we replace the pre-crossings, $p_i$, of Definition ~\ref{pmn} by the singular crossings $\tau_i$ in the relations, then we obtain the singular braid monoid, $SM_n$, defined in \cite{Ba, Bi}. Thus, we obtain the following result:

\begin{prop}[Proposition 2.3 \cite{BJW}]\label{propb1}
The pseudo braid monoid $PM_n$ is isomorphic to the singular braid monoid, $SM_n$. 
\end{prop}

\begin{remark}\label{propb2}\rm
In \cite{FKR} it is shown that the singular braid monoid $SM_n$ embeds in a group, the singular braid group $SB_n$. It follows that $PM_n$ embeds also in a group, the pseudo braid group $PB_n$. Obviously, $SB_n$ is isomorphic to $PB_n$ (Proposition 2.3 \cite{BJW}).
\end{remark}

Define now the {\it closure} of a pseudo braid as in the standard case (for an illustration see Figure~\ref{cl} by ignoring the ties). By considering the natural inclusion $PM_n \subset PM_{n+1}$, we may consider the inductive limit $PM_{\infty}$. Using the braiding algorithm for singular knots (\cite{Bi}), in \cite{BJW} the analogue of the Alexander theorem for pseudo knots is presented. In particular:

\begin{thm}[{\bf The analogue of the Alexander theorem for pseudo links}] \label{alexpl}
Every oriented pseudo knot is isotopic to the closure of a pseudo braid.
\end{thm}

Finally, we have a theorem for pseudo braid equivalence.

\begin{thm}[{\bf The analogue of the Markov Theorem for pseudo braids}] \label{markpl}
Two pseudo braids have isotopic closures if and only if one can  obtained from the other by a finite sequence of the following moves:

\[
\begin{array}{lllcll}
1. & {Conjugation:} &  \alpha & \sim & \beta^{\pm 1}\, \alpha\, \beta^{\mp 1}, & {\rm for}\ \alpha \in PM_n\ \&\ \beta \in B_n,\\
&&&&&\\
2. & {Commuting:} &  \alpha\, \beta & \sim & \beta\, \alpha, & {\rm for}\ \alpha,\, \beta \in PM_n,\\
&&&&&\\
3. & {Stabilization:} &  \alpha & \sim & \alpha\, \sigma_n^{\pm 1}, & {\rm for}\ \alpha \in PM_n,\\
&&&&&\\
4. & {Pseudo-stabilization:} &  \alpha & \sim & \alpha\, p_n, & \alpha \in PM_n.\\
\end{array}
\]
\end{thm}

In the proof of Theorem~\ref{markpl}, the fact that singular knots are closely related to pseudo knots is used, and the only interesting case is the case where two pseudo knots differ by a pseudo-Reidemeister move 1 (see left-most and right-most illustrations in Figure~\ref{mal}). More precisely, let $L, L^{\prime}$ be two pseudo knots that differ by a pseudo Reidemeister move 1 and let $l, l^{\prime}$ be the images of $L, L^{\prime}$ respectively under the isomorphism between singular braids and pseudo braids. By the analogue of the Alexander theorem for singular knots there are singular braids $\alpha, \alpha^{\prime}$, whose closures are singular isotopic to $l, l^{\prime}$ respectively. The images of these singular braids under the natural map between singular braids and pseudo braids are the pseudo braids $\beta, \beta^{\prime}$, whose closures are pseudo-isotopic to $L$ and $L^{\prime}$ respectively. In \cite{BJW} it is shown then that $\alpha$ differs from $\alpha^{\prime}$ by the Relation~4 of Theorem~\ref{markpl} and the proof is concluded. The above are summarized in the following diagram:

\[
\begin{array}{ccccccc}
&         & L             & \overset{{\bf PR1-moves}}{\longleftrightarrow}  & L^{\prime}     &     &      \\
&\swarrow & \uparrow     &                                                 & \uparrow       &  \searrow &  \\
l&       & \beta         &                                                 & \beta^{\prime} & &l^{\prime}  \\
&\nwarrow & \updownarrow  &                                                 & \updownarrow   & \nearrow  &   \\
&         & \alpha        &   \overset{{\bf Rel.(4)}}{\longleftrightarrow}
                                             & \alpha^{\prime}&         &     \\
\end{array}
\]

\begin{figure}[ht]
\begin{center}
\includegraphics[width=4.2in]{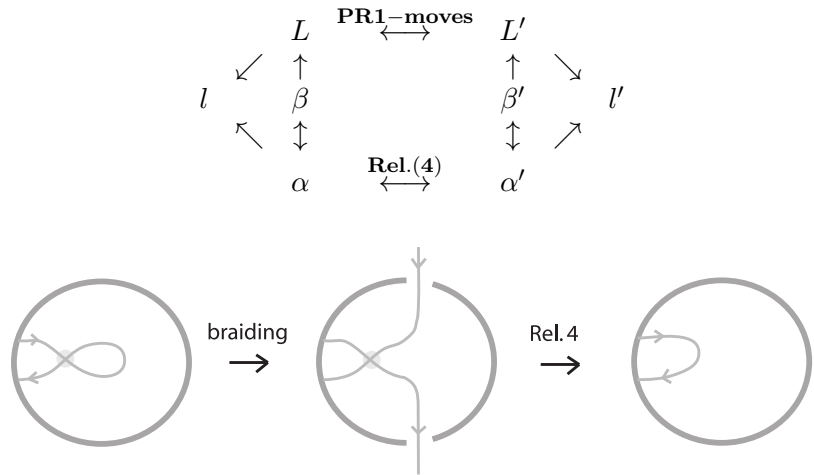}
\end{center}
\caption{RP1 and Rel. 4.}
\label{mal}
\end{figure}

\subsection{Knotoids \& Braidoids}\label{knbroid}

Knotoids were introduced by Turaev in \cite{T} as a generalization of 1-1 tangles by allowing the endpoints to be in different regions of the diagram. More precisely:

\begin{defn}\rm
A {\it knotoid diagram} $K$ in an oriented surface $\Sigma$ is a generic immersion of the unit interval $[0, 1]$ into  
$\Sigma$ whose only singularities are transversal double points endowed with over/undercrossing data called crossings. The images of $0$ and $1$ under this immersion are called the endpoints of $K$ (leg and head of $K$ respectively) and are distinct from each other and from the double points.
\end{defn}

A {\it knotoid} in $\Sigma$ is then an equivalence class of knotoid diagrams in $\Sigma$ up to the equivalence relation induced by the R1, R2, R3 moves (recall Figure~\ref{reid}) and planar isotopy. It is worth mentioning that we are not allowed to pull a strand that is adjacent to an endpoint, over or under a transversal arc, since this will result into a trivial knotoid diagram. These moves are illustrated in Figure~\ref{forb}, and they are called {\it forbidden moves} of knotoids.  

\begin{figure}[ht]
\begin{center}
\includegraphics[width=4.5in]{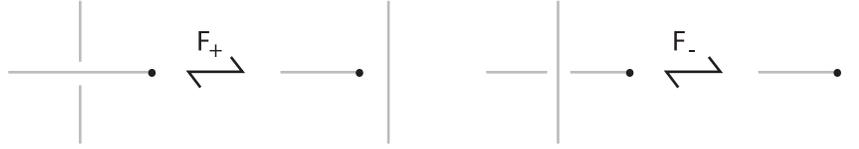}
\end{center}
\caption{The forbidden moves.}
\label{forb}
\end{figure}

Note that there are two situations where forbidden moves, seemingly occur as illustated in Figure~\ref{ffm}. We shall call these moves
{\it fake forbidden moves}.

\begin{figure}[ht]
\begin{center}
\includegraphics[width=5.6in]{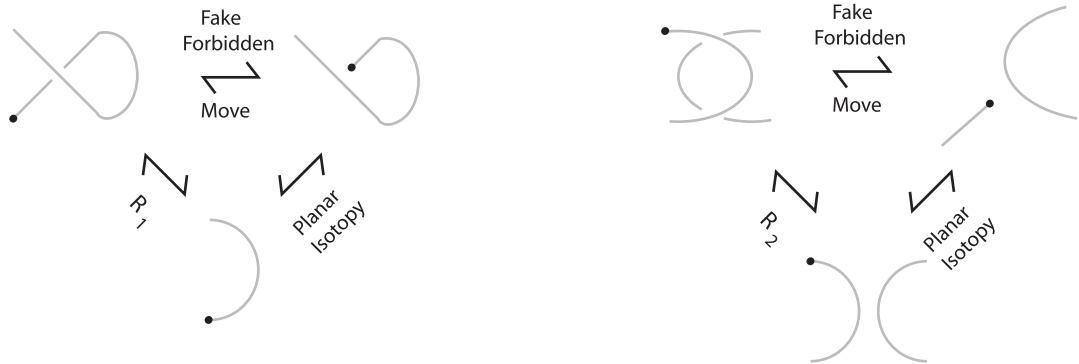}
\end{center}
\caption{Fake forbidden moves.}
\label{ffm}
\end{figure}

The definition of knotoids can be extended to linkoids and multi-knotoids as follows: a {\it linkoid diagram} is defined to be an immersion of a disjoint union of finitely many unit intervals whose images are knotoid diagrams and a {\it multi-knotoid diagram} is defined to be a union of a knotoid diagram and a finite number of knot diagrams. 

\smallbreak

Knotoid equivalence extends naturally to both linkoid diagrams and multi-knotoid diagrams, leading to the theory of linkoids and multi-knotoids. Finally, note that knotoids and linkoids are naturally oriented from the leg to the head.

\smallbreak

As noted in \cite{T}, the theory of knotoid diagrams suggests a new diagrammatic approach to knots, since every knotoid diagram determines a classical knot if we connect the endpoints of a knotoid diagram with an arc in $S^2$ that goes under or over each arc it meets. We call the resulting knots in $\mathbb{R}^3$, the {\it underpass closure} and the {\it overpass closure} of the knotoid respectively. Obviously, different closures of a knotoid may result in different knots. Thus, in order to represent knots via knotoid diagrams, we fix the closure type and we have the following result from \cite{T}:

\smallbreak

\begin{prop}
Assuming a specific closure type, there is a well-defined surjective map from knotoid diagrams to classical knots.
\end{prop}

In \cite{GL1} braidoid diagrams are defined (similarly to classical braid diagrams), as a system of finite descending strands that involves one or two strands starting with or terminating at an endpoint that is not necessarily at top or bottom lines of the defining region of the diagram. More precisely:

\begin{defn}\rm
A {\it braidoid diagram} $B$ is a system of a finite number of arcs immersed in $[0, 1] \times [0,1] \subset \mathbb{R}^2$, where $\mathbb{R}^2$ is identified with the xt-plane, such that the t-axis is directed downward. The arcs of $B$ are called the strands of $B$. Each strand of $B$ is naturally oriented downward, with no local maxima or minima, following the natural orientation of $[0, 1]$. Moreover, there are only finitely many intersection points among the strands, which are transversal double points endowed with over/under data, and are called crossings of $B$.
\end{defn}

A braidoid diagram has two types of strands, the classical strands, i.e. braid strands connecting points on $[0, 1]\times \{0\}$ to points on $[0, 1] \times \{1\}$, and the {\it free strands} that either connect a point in $[0,1]\times \{0\}$ or in $[0,1]\times \{1\}$ to an {\it endpoint} located anywhere in $[0, 1]\times [0, 1]$, or they connect two endpoints that are located anywhere in $[0, 1] \times [0, 1]$. These points that don't necessarily lie on $[0, 1] \times \{0\}$ or $[0, 1] \times \{1\}$, are called {\it braidoid ends}. For more details and examples the reader is referred to \cite{GL1}. We now present braidoid isotopy:

\begin{defn}\label{broidiso}\rm
Two braidoid diagrams are said to be isotopic if one can be obtained from the other by a finite sequence of the following moves, that we call  {\it braidoid isotopy} moves:
\smallbreak
\begin{itemize}
\item[$\bullet$] {\it Braidoid} $\Delta${\it -moves} illustrated in the left part of Figure~\ref{biso1}: a $\Delta$-move replaces a segment of a strand with two segments in a triangular disk free of endpoints, passing only over or under the arcs intersecting the triangular region of the move whilst the downward orientation of the strands is preserved.
\smallbreak
\item[$\bullet$] {\it Vertical moves} as illustrated in the right part of Figure~\ref{biso1}: the endpoints of a braidoid diagram can be pulled up or down in the vertical direction but without letting an endpoint of a braidoid diagram to be pushed/pulled over or under a strand (recall the forbidden moves).
\smallbreak

\begin{figure}[ht]
\begin{center}
\includegraphics[width=3.2in]{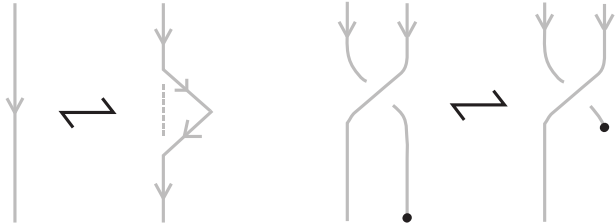}
\end{center}
\caption{A $\Delta$-move and a vertical move on a braidoid.}
\label{biso1}
\end{figure}

\item[$\bullet$] {\it Swing moves} as illustrated in Figure~\ref{biso2}: the endpoints are allowed to swing to the right or the left like a
pendulum as long as the downward orientation on the moving arc is preserved, and the forbidden moves are not violated.
\end{itemize}
\smallbreak
\begin{figure}[ht]
\begin{center}
\includegraphics[width=3.5in]{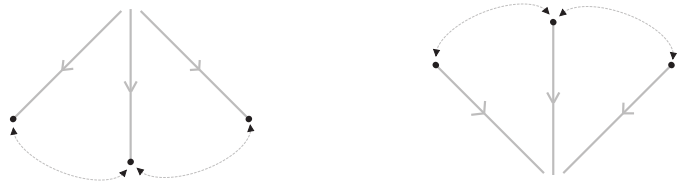}
\end{center}
\caption{The swing moves on braidoids.}
\label{biso2}
\end{figure}
\smallbreak
An isotopy class of braidoid diagrams is called a {\it braidoid}. Moreover, a {\it labeled braidoid diagram} is a braidoid diagram with a label over or under assigned to each pair of corresponding ends.
\end{defn}

\begin{defn}\rm
A {\it labeled braidoid diagram} is a braidoid diagram whose corresponding ends are labeled either with ``o'' or ``u'' in pairs. The {\it closure} of a labeled braidoid is realized by joining each pair of corresponding ends by a vertical segment, either over or under the rest of the braidoid and according to the label attached to these braidoid ends (see Figure~\ref{clab}).
\end{defn}

\begin{figure}[ht]
\begin{center}
\includegraphics[width=5.5in]{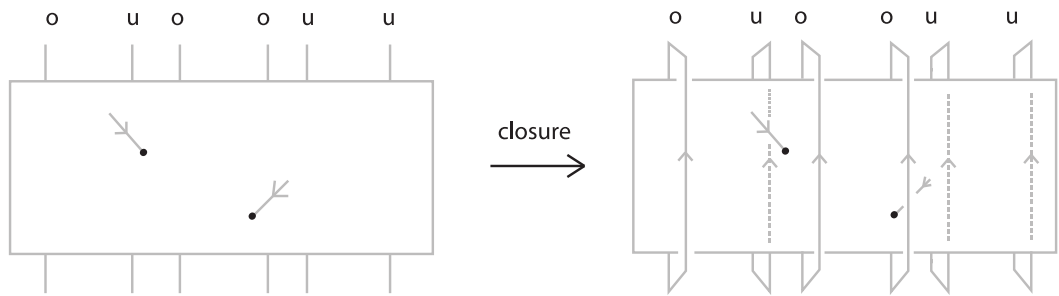}
\end{center}
\caption{The closure of a labeled braidoid.}
\label{clab}
\end{figure}

In \cite{GL1, GL2}, the authors present a braidoiding algorithm for knotoids, with the use of which, they obtain the following result:

\begin{thm}[{\bf The analogue of the Alexander theorem for knotoids}] \label{alexkn}
Any (multi)-knotoid diagram is isotopic to the closure of a (labeled) braidoid diagram.
\end{thm}

\begin{remark}\rm
It is crucial to note that different labels on the endpoints of a braidoid may yield non-equivalent closures. For more details the reader is referred to \cite{GL1}. Moreover, it is worth mentioning that in \cite{GL1} the authors prove that any knotoid diagram may be isotoped to be the closure of some labeled braidoid diagram whose labels are all ``u'' (\cite{GL1} Corollary 1), and they define a {\it uniform braidoid} to be a labeled braidoid with all labels ``u''.
\end{remark}

In \S~\ref{lmovepk} we recall the braiding algorithm and the $L$-moves from \cite{LR1} and we show that these notions can also be adopted for pseudo knots. In \cite{Gu}, [GL2] this braiding algorithm and the $L$-moves are adopted accordingly for the case of knotoids in order to prove a braidoid equivalence theorem. We recall these results in \S~\ref{pbrd} and we show how to adopt these notions for the case of pseudo knotoids.

\section{$L$-moves \& pseudo knots}\label{lmovepk}

In this section we introduce the notion of $L$-moves for pseudo knots in $S^3$ and we prove an $L$-move analogue of the Markov theorem for pseudo braids. We recall first the classical definition of an $L$-move (\cite{LR1}).

\begin{defn}\label{lmdefn}\rm
An {\it $L$-move} on a classical braid $\beta$, consists in cutting an arc of $\beta$ open and pulling the upper cutpoint downward and the lower upward, so as to create a new pair of braid strands with corresponding endpoints (on the vertical line of the cutpoint), and such that both strands cross entirely {\it over} or {\it under} with the rest of the braid. Stretching the new strands over will give rise to an {\it $L_o$-move\/} and under to an {\it  $L_u$-move\/} as shown in Figure~\ref{lm} by ignoring all pre-crossings.
\end{defn} 

In \cite{LR1} a sharpened version of the Markov theorem is proved using only the $L$-moves. $L$-moves make up an important tool for formulating braid equivalences in any topological setting and they prove to be particularly useful in settings where the sets of braid analogues do not have a ``nice'' algebraic structure. In \cite{LR1, La}, $L$-moves and braid equivalence theorems are presented for different knot theories. We now recall braid equivalence for singular knots.

\smallbreak

The $L$-moves for singular braid equivalence are defined \cite{La} as in the classical case (Definition~\ref{lmdefn}), that is, the two strands that appear after the performance of an $L$-move should cross the rest of the braid only with real crossings (all over in the case of an $L_o$-move or all under in the case of an $L_u$-move). Finally, with the use of the $L$-moves, in \cite{La}, a sharpened version of the analogue of the Markov theorem (compared to \cite{G}) for singular braids is presented:

\begin{thm}[{\bf The analogue of the Markov Theorem for singular links}]\label{mthsll}
Two oriented singular links are isotopic if and only if any two corresponding singular braids differ by braid relations in $SB_{\infty}$ and a finite sequence of the following moves:
\[
\begin{array}{rlccc}
{\rm i.} & {Singular\ commuting:} & \tau_i\, \alpha & \sim & \alpha\, \tau_i\\
&&&&\\
{\rm ii.} & {L-moves}, & &&\\
\end{array}
\]
\smallbreak
\noindent where $\alpha \in SB_n$, the singular braid group, and $\tau_i$ a singular crossing.
\end{thm}

\subsection{$L$-moves in pseudo braids}\label{lpb}

Our intention is to present an $L$-move braid equivalence theorem for pseudo knots. We first need to define $L$-moves for pseudo braids. $L$-moves for pseudo braids are defined in the same way as in the singular case (for an illustration see Figure~\ref{lm}). 

\begin{defn}\rm
An {\it $L$-move} on a pseudo braid $\beta$, consists in cutting an arc of $\beta$ open and pulling the upper cutpoint downward and the lower upward, so as to create a new pair of braid strands with corresponding endpoints (on the vertical line of the cutpoint), and such that both strands cross the rest of the braid only with real crossings (all over in the case of an $L_o$-move or all under in the case of an $L_u$-move). When applying an $L$-move on a pseudo braid, an inbox classical crossing or a pre-crossing may be introduced as illustrated in Figure~\ref{plm} (recall move PR1 in Figure~\ref{reid}).
\end{defn}

\begin{figure}[ht]
\begin{center}
\includegraphics[width=4.9in]{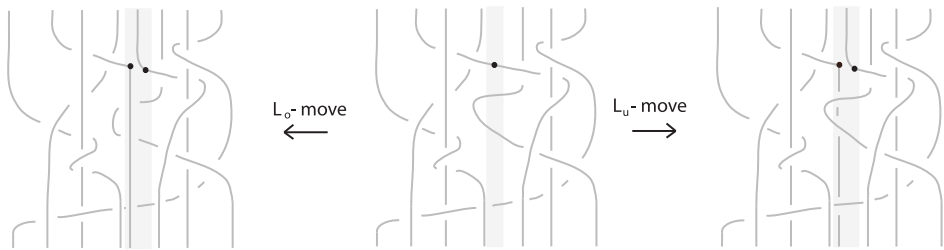}
\end{center}
\caption{$L$-moves for pseudo braids.}
\label{lm}
\end{figure}

\begin{remark}\rm
Note that in the case of singular braids the introduction of an inbox-singular crossing is not allowed when performing an $L$-move, since the singular kink in the contraction would never be eliminated, due to the forbidden Reidemeister 1 move. {\it This constitutes the main difference between the theory of pseudo braids and the theory of singular braids.}
\end{remark}

\begin{figure}[H]
\begin{center}
\includegraphics[width=4.9in]{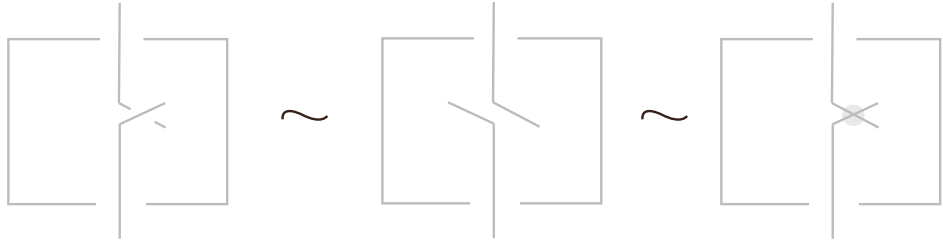}
\end{center}
\caption{An $L$-move in a pseudo braid introducing a real and a pseudo crossing.}
\label{plm}
\end{figure}

\subsection{$L$-move braiding algorithm for pseudo links}\label{alse}

In the classical setting the counterpart of the $L$-moves on oriented link diagrams are the {\it braiding moves}, as illustrated abstractly in Figure~\ref{ahg}, which were used in \cite{LR1} for braiding any oriented link diagram. The main idea is to keep the arcs of the oriented link diagrams that go downwards with respect to the height function unaffected and replace arcs that go upwards with braid strands. These arcs are called {\it up-arcs} (see Figure~\ref{upa}).

\begin{figure}[ht]
\begin{center}
\includegraphics[width=4.4in]{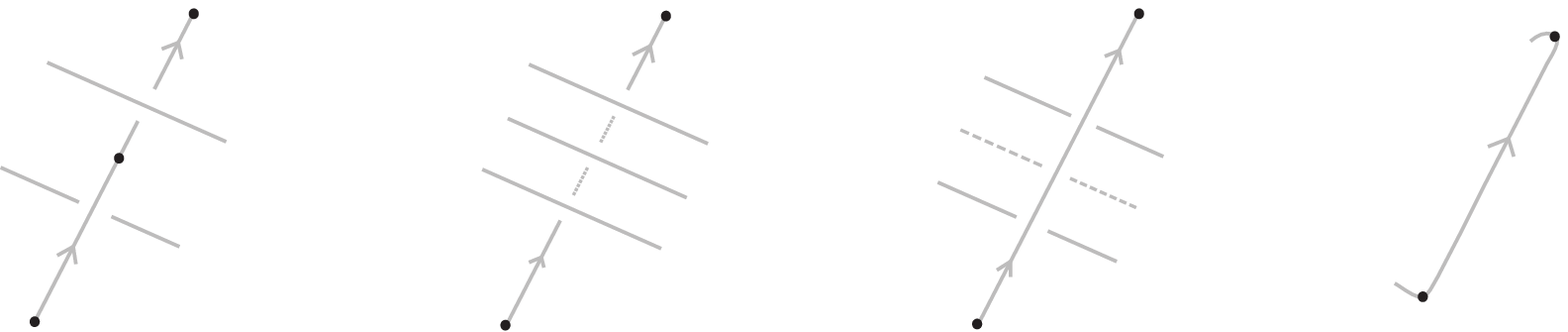}
\end{center}
\caption{Up-arcs.}
\label{upa}
\end{figure}

In order to use this classical braiding algorithm on pseudo knots we need to also deal with pre-crossings in the diagram which contain at least one up-arc. For this we apply the idea used in \cite{Bi} for the case of singular knots, a detailed proof of which can be found in \cite{PR} Theorem~2.3 (see also \cite{KL} for the case of virtual knots). Namely, before we apply the braiding algorithm we have to isotope the pseudo link in such a way that the pre-crossings will only contain down-arcs, so that the braiding algorithm will not affect them. This is achieved by rotating all pre-crossings that contain at least one up-arc, so that the two arcs are now directed downward, as illustrated in Figure~\ref{tw}. Then we may apply the braiding algorithm of \cite{LR1} for the knot (ignoring the pre-crossings).

\begin{figure}[ht]
\begin{center}
\includegraphics[width=4.7in]{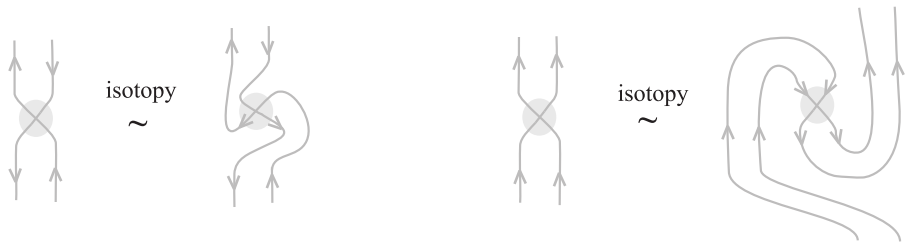}
\end{center}
\caption{Rotating pre-crossings.}
\label{tw}
\end{figure}

The algorithm is summarized as follows:

\bigbreak

\begin{itemize}
\item We first isotope the diagram of the pseudo link as described above. Then, we apply the following braiding algorithm for the knot:
\smallbreak
\item We chose a base-point and we run along the diagram of the (pseudo) link according to its orientation.
\smallbreak
\item When/If we run along an opposite arc, we subdivide it into smaller arcs, each containing crossings of one type only as shown in Figure~\ref{upa}.
\smallbreak
\item We now label every up-arc with an ``o''or a ``u'', according to the crossings it contains. If it contains no crossings, then the choice is arbitrary.
\smallbreak
\item We perform an $o$-braiding moves on all up-arcs which were labeled with an ``o'' and $u$-braiding moves on all up-arcs which were labeled with an ``u'' (see Figure~\ref{ahg}).
\smallbreak

\begin{figure}[ht]
\begin{center}
\includegraphics[width=2.4in]{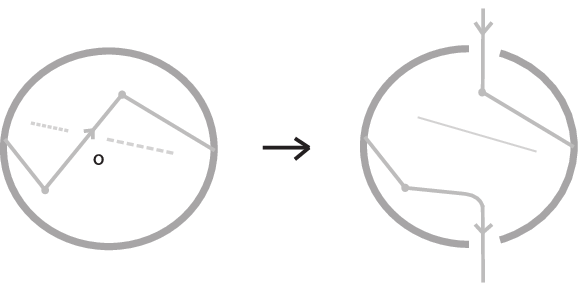}
\end{center}
\caption{Braiding moves for up-arcs.}
\label{ahg}
\end{figure}

\item The result is a pseudo braid whose closure is isotopic to the initial pseudo link.

\end{itemize}

The above algorithm provides a proof of the following:

\begin{thm}[{\bf The analogue of the Alexander theorem for pseudo knots}]\label{newprpkalex}
Every oriented pseudo knot is isotopic to the closure of a pseudo braid.
\end{thm}

\subsection{$L$-move equivalence for pseudo braids}\label{lmpb}

As shown in \cite{LR1} (see also \cite{La} for the case of singular braids), the $L$-moves can realize (real) conjugation, while (real) stabilization moves are special cases of $L$-moves. Furthermore, an $L$-move on a pseudo braid that introduces a pre-crossing (recall Figure~\ref{plm}) has pseudo-stabilization as a special case. Finally, $L$-moves cannot achieve commuting by a pre-crossing (see Figure~\ref{pcsing}),
which is allowed in the theory.

\begin{figure}[ht]
\begin{center}
\includegraphics[width=3.4in]{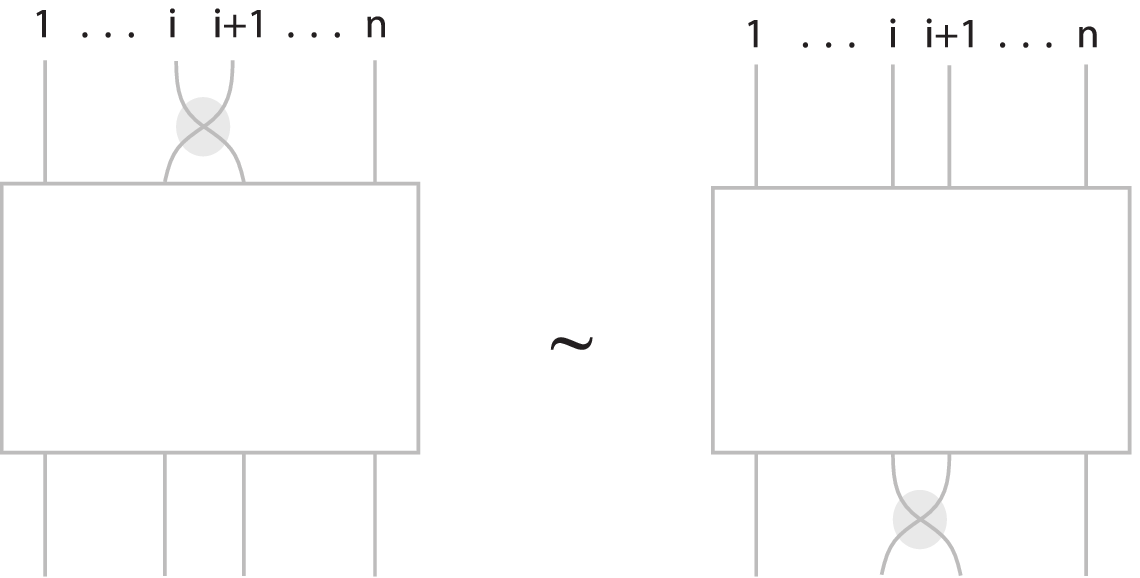}
\end{center}
\caption{Braiding moves for up-arcs.}
\label{pcsing}
\end{figure}

From the discussion above, it follows that we may replace moves 2, 3 and 4 in Theorem~\ref{markpl}, by $L$-moves and we obtain the following:

\begin{thm}[{\bf $L$-move equivalence for pseudo braids}] \label{lmarkpl}
Two pseudo braids have isotopic closures if and only if one can  obtained from the other by a finite sequence of the following moves:
\[
\begin{array}{lllcll}
1. & {L-moves}\ &   &  &   & \\
&&&&&\\
2. & {Commuting:} &  \alpha\, p_i & \sim & p_i\, \alpha, & {\rm for}\ \alpha,\, \beta \in PM_n.\\
\end{array}
\]
\end{thm}

\section{Tied Pseudo links \& Tied Singular links}\label{tpl}

In this section we introduce and study {\it tied pseudo links} and {\it tied pseudo braids} with the use of tied singular links and tied singular braids, introduced and studied in \cite{AJ3}. We also introduce the tied pseudo braid monoid and we conclude by proving the analogue of the Alexander and Markov theorems for tied pseudo links.

\smallbreak

\begin{defn}\rm
A {\it tied singular link} is a singular link with ties, while a {\it tied pseudo knot} is a pseudo knot equipped with ties.
\end{defn}

\begin{defn}\label{tsmn}\rm
The {\it tied singular braid monoid} $TSM_n$ is defined as the monoid generated by the standard braid generators $\sigma_i^{\pm 1}$'s, the singular braiding generators $\tau_{i}$'s of $SM_n$ and the ties generators $\eta_i$'s of $TM_n$ satisfying the defining relations of $SM_n$, the defining relations of $TM_n$ together with the following relations:

\[
\begin{array}{rcll}
\tau_i\, \eta_i & = & \eta_i\, \tau_i, & {\rm for\ all}\ i,\\
&&&\\
\tau_i\, \eta_j & = & \eta_j\, \tau_i, & {\rm for}\ |i-j|\geq 2,\\
&&&\\
\eta_i\, \tau_j\, \tau_i & = & \tau_j\, \tau_i\, \eta_j, & {\rm for}\ |i-j|=1,\\
&&&\\
\eta_i\, \eta_j\, \tau_i & = & \eta_j\, \tau_i\, \eta_j \ =\ \tau_i\, \eta_i\, \eta_j, & {\rm for}\ |i-j|=1,\\
&&&\\
\eta_i\, \tau_j\, \sigma_i & = & \tau_j\, \sigma_i\, \eta_j, & {\rm for}\ |i-j|=1,\\
&&&\\
\eta_i\, \sigma_j\, \tau_i & = & \sigma_j\, \tau_i\, \eta_j, & {\rm for}\ |i-j|=1,\\
&&&\\
\tau_i\, \eta_j & = & \sigma_i\, \eta_j\, \sigma_i^{-1}\, \tau_i, & {\rm for}\ |i-j|=1.\\
\end{array}
\]
\end{defn}

To define the monoid of tied pseudo braids $PM_n$, we consider the generators $\sigma_i^{\pm 1}$ of the braid group $B_n$, the ties $\eta_i$ of $TM_n$ and we add the generators $p_i$, replacing the $\tau_i$'s. We need to find the defining relations of this monoid and for this reason we analyze the pseudo-Reidemeister moves and obtain the pseudo braid monoid (see \cite{BJW}) and we also analyze relations concerning ties and pre-crossings, leading to the following definition:

\begin{defn}\label{tpmn}\rm
The {\it tied pseudo braid monoid} $TPM_n$ is defined as the monoid generated by the standard braid generators $\sigma_i^{\pm 1}$'s, the pseudo-generators $p_{i}$'s of $SM_n$ and the ties generators $\eta_i$'s of $TM_n$ satisfying the defining relations of $PM_n$, the defining relations of $TM_n$ together with the following relations:

\[
\begin{array}{rcll}
p_i\, \eta_i & = & \eta_i\, p_i, & {\rm for\ all}\ i,\\
&&&\\
p_i\, \eta_j & = & \eta_j\, p_i, & {\rm for}\ |i-j|\geq 2,\\
&&&\\
\eta_i\, p_j\, p_i & = & p_j\, p_i\, \eta_j, & {\rm for}\ |i-j|=1,\\
&&&\\
\eta_i\, \eta_j\, p_i & = & \eta_j\, p_i\, \eta_j \ =\ p_i\, \eta_i\, \eta_j, & {\rm for}\ |i-j|=1,\\
\end{array}
\]

\[
\begin{array}{rcll}
\eta_i\, p_j\, \sigma_i & = & p_j\, \sigma_i\, \eta_j, & {\rm for}\ |i-j|=1,\\
&&&\\
\eta_i\, \sigma_j\, p_i & = & \sigma_j\, p_i\, \eta_j, & {\rm for}\ |i-j|=1,\\
&&&\\
p_i\, \eta_j & = & \sigma_i\, \eta_j\, \sigma_i^{-1}\, p_i, & {\rm for}\ |i-j|=1.\\
\end{array}
\]
\end{defn}

Comparing Definitions~\ref{tsmn} \& \ref{tpmn}, we have the following:

\begin{thm}\label{is1}\rm
There exists an isomorphism $\mu$ from the tied singular braid monoid to the tied pseudo braid monoid, defined as follows:
\bigbreak
\begin{equation}\label{iso}
\begin{array}{rccl}
\mu\, : & TSM_n & \rightarrow & TPM_n\\
&&&\\
        &  \sigma_i^{\pm 1} & \mapsto & \sigma_i^{\pm 1}\\
				&  \tau_i           & \mapsto & p_i\\
\end{array}
\end{equation}
\end{thm}

Moreover, in Theorem~9 \cite{AJ3} it is shown that the monoid $TSM_n$ is isomorphic to $P_n \rtimes SM_n$, where $P_n$ is the monoid generated by the set of partitions generated by the ties. This separates the ties from the rest of the singular braid, i.e. in a tied singular braid one can bring the ties at the top or bottom of the tied singular braid. Since now $TPM_n$ is isomorphic to $TSM_n$ and $TM_n$ is isomorphic to $SM_n$, it is natural to consider $TPM_n$ to be isomorphic to $P_n \rtimes TM_n$. A proof of this fact is obtained by recalling that generalized ties $\eta_{i, j}$ are transparent with respect to all strands between the $i^{th}$ and $j^{th}$ strands, and that generalized ties are provided with elasticity.

\smallbreak

Finally, the generalized ties satisfy the following relations (recall Eq.~\ref{eqf}):

\begin{equation}\label{eqf1}
\begin{array}{crclcrcl}
(i) & p_i\, \eta_{i, j} & = & \eta_{i+1}\, p_i, & (ii) & p_j\, \eta_{i, j} & = & \eta_{i, j+1}\, p_i,\\
&&&&&&&\\
(iii) & p_{i-1}\, \eta_{i, j} & = & \eta_{i-1, j}\, p_{i-1}, & (iv) & p_{j-1}\, \eta_{i, j} & = & \eta_{i, j-1}\, p_{j-1}, \\
\end{array}
\end{equation}

\noindent which, together with the transparent property \& elasticity of ties, lead to the following results:

\begin{thm}\label{mobtpb}
Let $\alpha$ be a tied pseudo braid in $TPM_n$. Then, $\alpha$ can be written as $\alpha\, =\, \gamma\, \beta$ (or $\alpha\, =\, \beta\, \gamma$), where $\gamma\in PM_n$ and $\beta\in TPM_n^{\sim}$, where $TPM_n^{\sim}$ denotes the subset of $TPM_n$ generated by the ties.
\end{thm}

Moreover:

\begin{prop}\rm \label{propaj1}
Let $\alpha_1=\beta_1\, \gamma_1,\ \alpha_2=\beta_2\, \gamma_2$ be two tied pseudo braids in $TPM_n$. Then, $\alpha_1\, \sim\, \alpha_2$, if and only if $\beta_1\, \sim\, \beta_2$ in $TM_n$ and $\gamma_1, \gamma_2$ define the same partition of the set of the strands. 
\end{prop}

\subsection{Braiding tied pseudo links}\label{tpbal}

Define now the {\it closure} of a tied pseudo braid as in the classical case by ignoring the ties and the pre-crossings. For an illustration see Figure~\ref{cl}.

\begin{figure}[ht]
\begin{center}
\includegraphics[width=2.9in]{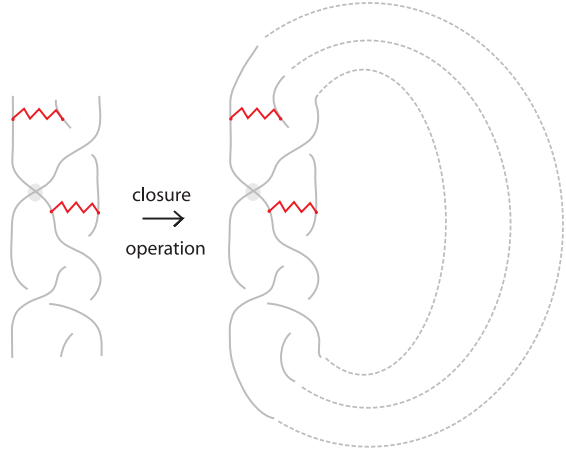}
\end{center}
\caption{The closure operation for tied pseudo braids.}
\label{cl}
\end{figure}

We are now in position to prove the analogue of the Alexander theorem for tied pseudo braids.

\begin{thm}[{\bf The analogue of the Alexander theorem for tied pseudo links}]\label{alexthmtpl}
Every oriented tied pseudo link is equivalent to the closure of a tied pseudo braid.
\end{thm}

\begin{proof}
The proof follows by using the braiding algorithm of \S~\ref{lmovepk} and the fact that ties can freely move along strands of the pseudo links (see also \cite{D} for braiding algorithms for tied links in other c.c.o. 3-manifolds).
\end{proof}

\begin{remark}\rm
An alternative proof follows in a similar way as in the proof of Theorem~3.1 in \cite{BJW}. More precisely, let $L$ be a tied pseudo link and let $f$ be the natural map from the tied singular links to tied pseudo links. Consider the tied singular link $L^{\prime}$ in $f^{-1}(L)$. From Theorem~10 in \cite{AJ3} we have that there exists a tied singular braid $\beta^{\prime}$, whose closure is equivalent to $L^{\prime}$. From Theorem~\ref{is1} and Eq.~\ref{iso}, there exist a tied pseudo braid $\beta$ whose closure is the tied pseudo link $L$.
\end{remark}

\subsection{Braid equivalence for tied pseudo braids}\label{mthmtpb}

We begin this subsection by recalling the analogue of the Markov theorem for tied singular braids from \cite{AJ3}.

\begin{thm}[{\bf The analogue of the Markov theorem for tied singular braids}] \label{marktsl}
Two tied singular braids have equivalent closures if and only if one can  obtained from the other by a finite sequence of the following moves:

\[
\begin{array}{lllcll}
1^{\prime}. & {Conjugation:} &  \alpha\, \beta & \sim & \beta\, \alpha, & {\rm for\ all}\ \alpha \in TSM_n\ \&\ \beta\in TM_n \\
&&&&&\\
2^{\prime}. & {Commuting:} &  \alpha\, \tau_i & \sim & \tau_i \, \alpha, & {\rm for\ all}\ \alpha \in TSM_n,\\
&&&&&\\
3^{\prime}. & {Stabilization:} &  \alpha & \sim & \alpha\, \sigma_n^{\pm 1}, & {\rm for\ all}\ \alpha \in TSM_n,\\
&&&&&\\
4^{\prime}. & {t-stabilization:} &  \alpha & \sim & \alpha\, \eta_{i, j}, & {\rm for\ all}\ \alpha \in TSM_n\ {\rm such\ that}\\
\end{array}
\]
\smallbreak
\noindent $s_{\alpha}(i)=j$, {\rm where} $s_{\alpha}$ {denotes the permutation associated to the braid obtained from} $\alpha$ {by forgetting its ties.}
\end{thm}

\begin{remark}\rm
It is worth mentioning that Theorem~\ref{marktsl} is stated and proved in \cite{AJ3} in the context of {\it combinatoric tied singular links}. Moreover, in this paper we prefer to distinguish conjugation of classical invertible generators, and commuting of non invertible generators, which are the singular crossings and the elementary ties.
\end{remark}

Note that from the discussion above it follows that moves $1^{\prime}$ and $3^{\prime}$ of Theorem~\ref{marktsl} can be replaced by the $L$-moves, where $L$-moves for tied braids are defined as in the classical case. This leads to an $L$-move braid equivalence theorem for tied singular braids:

\begin{thm}[{\bf $L$-move equivalence for tied singular braids}] \label{lmarktsl}
Two tied singular braids have equivalent closures if and only if one can  obtained from the other by a finite sequence of the following moves:
\[
\begin{array}{llcll}
{L-moves} &   &  &  & \\
&&&&\\
{Commuting:} &  \alpha\, \tau_i & \sim & \tau_i \, \alpha, & {\rm for\ all}\ \alpha \in TSM_n,\\
&&&&\\
{t-stabilization:} &  \alpha & \sim & \alpha\, \eta_{i, j}, & {\rm for\ all}\ \alpha \in TSM_n\ {\rm such\ that}\\
\end{array}
\]
\smallbreak
\noindent $s_{\alpha}(i)=j$, {\rm where} $s_{\alpha}$ {denotes the permutation associated to the braid obtained from} $\alpha$ {by forgetting its ties.}
\end{thm}

As noted before, we may consider the theory of tied pseudo links as the quotient of the theory of tied singular links by the pseudo-Reidemeister move 1. Thus, in order to obtain a braid equivalence theorem for tied pseudo braids, we need to consider moves of Theorem~\ref{marktsl}, together with the {\it pseudo-stabilization} moves of Theorem~\ref{markpl}.

\begin{thm}[{\bf The analogue of the Markov theorem for tied pseudo braids}] \label{marktpb}
Two tied pseudo braids have equivalent closures if and only if one can  obtained from the other by a finite sequence of the following moves:

\[
\begin{array}{lllcll}
1^{\prime \prime}. & {Commuting:} &  \alpha\, p_i & \sim & p_i\, \alpha, & {\rm for\ all}\ \alpha\, \in TPM_n,\\
&&&&&\\
2^{\prime \prime}. & {Conjugation:} & \beta & \sim & \alpha^{\pm 1}\, \beta\, \alpha^{\mp 1} & {\rm for\ all}\ \beta\in TPM_n\ \&\ \alpha\in B_n,\\
&&&&&\\
3^{\prime \prime}. & {Real-stabilization:} &  \alpha & \sim & \alpha\, \sigma_n^{\pm 1}, & {\rm for\ all}\ \alpha \in TPM_n,\\
&&&&&\\
4^{\prime \prime}. & {Pseudo-stabilization:} &  \alpha & \sim & \alpha\, p_n, & {\rm for\ all}\ \alpha \in TPM_n.\\
&&&&&\\
5^{\prime \prime}. & {t-stabilization:} &  \alpha & \sim & \alpha\, \eta_{i, j}, & {\rm for\ all}\ \alpha \in TPM_n\ {\rm such\ that}\\
\end{array}
\]

\noindent $s_{\alpha}(i)=j$, {\rm where} $s_{\alpha}$ {denotes the permutation associated to the braid obtained from} $\alpha$ {by forgetting its ties.}
\end{thm}

We now proceed with the proof of Theorem~\ref{marktpb}.

\begin{proof}
Obviously, if two tied pseudo braids differ by moves of Theorem~\ref{marktpb}, then their closures are equivalent. For the converse, let $L, L^{\prime}$ be two tied pseudo links and let $l, l^{\prime}$ be the corresponding tied singular links obtained from $L$ and $L^{\prime}$ respectively, by changing pre-crossings to singular crossings. From the braiding algorithm for tied singular braids, we may braid $l$ and $l^{\prime}$ to $\alpha$ and $\alpha^{\prime}$ respectively. By Theorem~\ref{is1}, we may assume that $\alpha\, =\, \mu(\beta)$ and $\alpha^{\prime}\, =\, \mu(\beta^{\prime})$, where $\mu$ is the isomorphism between $TSM_n$ and $TPM_n$. If now $L$ differs from $L^{\prime}$ by any isotopy move except from pseudo Reidemeister moves 1, then by Theorem~\ref{marktsl}, the tie singular braids $\alpha, \alpha^{\prime}$ differ by moves $1^{\prime}, 2^{\prime}, 3^{\prime}, 4^{\prime}$ and thus, $\beta, \beta^{\prime}$ differ by moves $1^{\prime \prime}, 2^{\prime \prime}, 3^{\prime \prime}$ and $5^{\prime \prime}$.

\smallbreak

If now $L, L^{\prime}$ differ by a pseudo Rreidemeister 1 move, we follow the proof of Theorem~4.2 in \cite{BJW} keeping in mind that the endpoints of the ties can freely move along strands. In that way, we may always reduce the complexity that ties may create in the diagrams. Without loss of generality, we place now $L$ and $L^{\prime}$ in $S^2$ in such a way that their diagrams are identical except for a region as shown in top part of Figure~\ref{alm1}. We now perform the braiding algorithm from \S~\ref{lmovepk} on all identical up-arcs of $L$ and $L^{\prime}$, leaving the region that the tied pseudo links differ at, in the end. Obviously, the tied pseudo braids $\beta$ and $\beta^{\prime}$ differ by a $4^{\prime \prime}$ move.

\begin{figure}[ht]
\begin{center}
\includegraphics[width=2.5in]{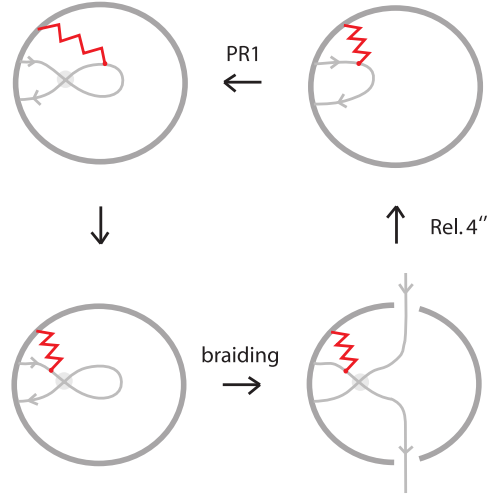}
\end{center}
\caption{RP1 and Rel. 4''.}
\label{alm1}
\end{figure}

\end{proof}

\begin{remark}\rm
An alternative proof of Theorem~\ref{marktpb} can be obtained by using Theorem~\ref{mobtpb} in order to separate the ties from the rest of the braiding generators of a tied pseudo braid and then follow the proof of Theorem~13 in \cite{D} by using generalized ties and properties that generalized ties satisfy. Namely, if $L$ and $L^{\prime}$ are two isotopic tied pseudo links and $\beta, \beta^{\prime}$ corresponding tied pseudo braids of theirs, then from Theorem~\ref{mobtpb} we have that $\beta=\gamma\, k$ and $\beta^{\prime}= \gamma^{\prime}\, k^{\prime}$, where $\gamma, \gamma^{\prime} \in PM_n$ and $k, k^{\prime} \in TPM_{n}^{\sim}$. Ignoring the ties in the tied pseudo links $L$ and $L^{\prime}$ results into two isotopic pseudo links. Thus, corresponding pseudo braids of theirs, $\alpha, \alpha^{\prime}$, are equivalent and one may obtain $\alpha^{\prime}$ from $\alpha$ via a sequence of moves in Theorem~\ref{markpl} or Theorem~\ref{lmarkpl}. Note that $\alpha$ and $\alpha^{\prime}$ are obtained from $\beta, \beta^{\prime}$ respectively, by ignoring the ties. Thus, $\gamma$ is equivalent to $\gamma^{\prime}$ by a sequence of Theorem~\ref{markpl} moves. From Proposition~\ref{propaj1}, it suffices to show now that $k$, $k^{\prime}$ define the same partition of components of the links. This is achieved in the same way as in \cite{AJ1} for tied links in $S^3$, as in \cite{F} for tied links in the Solid Torus and as in \cite{D} for tied links in various 3-manifolds.
\end{remark}

By similar argumentation as before, it follows that the moves $2^{\prime \prime}, 3^{\prime \prime}$ and $4^{\prime \prime}$ of Theorem~\ref{marktpb} can be replaced by the $L$-moves. This leads to the $L$-move analogue of Theorem~\ref{marktpb}:

\begin{thm}[{\bf Braid equivalence for tied pseudo braids via $L$-moves}] \label{marktpb2}
Two tied pseudo braids have equivalent closures if and only if one can  obtained from the other by a finite sequence of the following moves:
\[
\begin{array}{llcll}
{L-moves} &   & & & \\
&&&&\\
{Commuting:} & \alpha\, p_i & \sim & p_i\, \alpha, & {\rm for\ all}\ \alpha\, \in TPM_n,\\
&&&&\\
{t-stabilizations:} &  \alpha & \sim & \alpha\, \eta_{i, j}, & {\rm for\ all}\ \alpha \in TPM_n\ {\rm such\ that}\\
\end{array}
\]
\noindent $s_{\alpha}(i)=j$, {\rm where} $s_{\alpha}$ {denotes the permutation associated to the braid obtained from} $\alpha$ {by forgetting its ties.}
\end{thm}

\section{Pseudo knotoids \& Pseudo braidoids}\label{pknoid}

In this section we introduce and study the theory of {\it pseudo knotoids} and {\it pseudo braidoids}. We present a braiding algorithm for pseudo knotoids and we also introduce the notion of $L$-moves for pseudo braidoids following \cite{GL1, GL2}, which allow us to obtain the analogue of the Markov theorem for pseudo braidoids.

\subsection{Pseudo knotoids}\label{pkoidnew}

\begin{defn}\rm
A {\it pseudo knotoid diagram} in a surface $\Sigma$ is a knotoid diagram in $\Sigma$ where some crossing information may be missing. The undetermined crossings are again called {\it pre-crossings} (for an illustration see the left part of Figure~\ref{pcl}).
\end{defn}

A pseudo knotoid in $\Sigma$ is then an equivalence class of pseudo knotoid diagrams in $\Sigma$ up to the equivalence relation induced by moves illustrated in Figure~\ref{reid} and planar isotopy. As illustrated in Figure~\ref{pfm}, there is an extra forbidden move in the case of pseudo knotoids, similar to the first forbidden move in the case of knotoids, tha we call {\it pseudo forbidden move}.

\begin{figure}[ht]
\begin{center}
\includegraphics[width=2in]{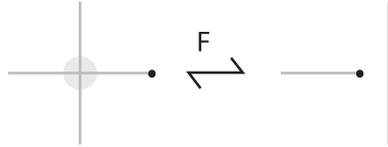}
\end{center}
\caption{The pseudo forbidden move.}
\label{pfm}
\end{figure}

Note that there is a situation where pseudo forbidden moves, seemingly occur as illustated in Figure~\ref{πffm}. We shall call these moves
{\it fake pseudo forbidden moves}.

\begin{figure}[ht]
\begin{center}
\includegraphics[width=2.2in]{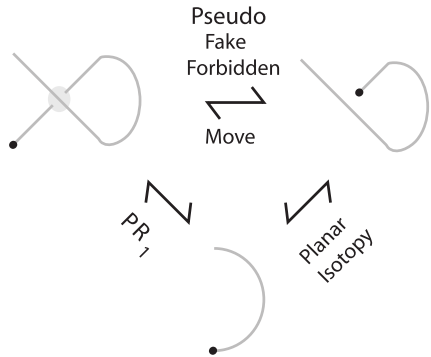}
\end{center}
\caption{The fake pseudo forbidden move.}
\label{πffm}
\end{figure}

\begin{remark}\rm
In \cite{MLK} the theory of {\it singular knotoids} is introduced. In particular, singular knotoid diagrams are defined as knotoid diagrams in $\Sigma$ that contain a finite number of self-intersections. Pseudo knotoids constitute the corresponding theory of singular knotoids, in a similar manner, as singular knot theory is related to pseudo knot theory. 
\end{remark}

In analogy to the definition of singular closure in \cite{MLK} (under certain conditions) one can also define the {\it pseudo closure} of knotoids as follows: 
 
\begin{defn}\rm
We call {\it pseudo closure} for knotoids, a mapping from the set of knotoid diagrams in a surface $\Sigma$ to the set of pseudo knot diagrams in $\Sigma$, where the endpoints of a knotoid diagram are connected with an embedded arc in $\Sigma$ and a pre-crossing is created every time the connection arc crosses a strand of the diagram.
\end{defn}

\begin{figure}[ht]
\begin{center}
\includegraphics[width=4.3in]{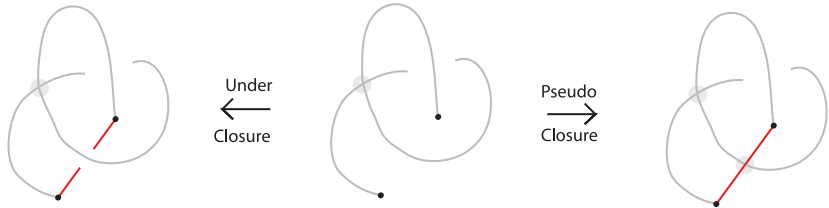}
\end{center}
\caption{The under closure and the pseudo closure of a pseudo knotoid.}
\label{pcl}
\end{figure}

\subsection{Pseudo braidoids}\label{pbrd}

Similarly to {\it braidoids}, one may define {\it pseudo braidoids}, the counterpart of pseudo knotoids, to be (labeled) braidoids with some crossing information missing. In order to obtain pseudo (labeled) braidoid isotopy, we allow the analogue of the vertical move on pseudo braidoids, that we call {\it pseudo vertical  move}, as illustrated in Figure~\ref{biso3}.

\begin{figure}[ht]
\begin{center}
\includegraphics[width=1.35in]{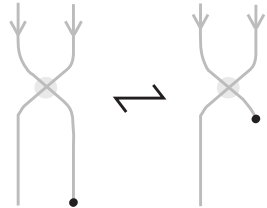}
\end{center}
\caption{The pseudo vertical move.}
\label{biso3}
\end{figure}

In Figure~\ref{fmb2} the forbidden move on pseudo braidoids that we call pseudo forbidden move is illustrated.

\begin{figure}[ht]
\begin{center}
\includegraphics[width=1.8in]{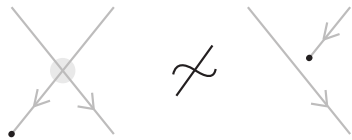}
\end{center}
\caption{The pseudo forbidden move on pseudo braidoids.}
\label{fmb2}
\end{figure}

\begin{defn}\label{pbroidiso}\rm
Two pseudo braidoid diagrams are said to be isotopic if one can be obtained from the other by a finite sequence of the moves in Definition~\ref{broidiso} together with the pseudo vertical move. Note again that only restricted swing moves are allowed in isotopy of pseudo braidoid diagrams. Moreover, an isotopy class of pseudo braidoid diagrams is called a {\it pseudo braidoid} and finally, a {\it labeled pseudo braidoid diagram} is a pseudo braidoid diagram with a label over or under assigned to each pair of corresponding ends.
\end{defn}

We now define the closure operation of pseudo braidoids in the same way as in the case of classical braidoids, and we have the following result:

\begin{thm}[{\bf The analogue of the Alexander theorem for pseudo knotoids}] \label{alexpknoid}
Every pseudo (multi)-knotoid can be obtained by closing a pseudo braidoid.
\end{thm}

\begin{proof}
The proof is similar to the braiding algorithm presented in \cite{LR1, Gu, GL2}. The only difference lies in the fact that before we apply the algorithm, we need to isotope the pseudo knotoid so that all pre-crossings are oriented downward as done in the case of pseudo knots (recall the algorithm in \S~\ref{alse}). In this way we guarantee that the braidoiding process will not affect the pre-crossings and the result follows.
\end{proof}

\begin{remark}\rm
The reader is referred to \cite{Gu, GL2} for details regarding the braidoiding algorithms for knotoids and how all steps of the braiding algorithm in \cite{LR1} are adopted in the case of knotoids.
\end{remark}

Before we define $L$-moves for labeled pseudo braidoid diagrams, we recall the definition of the $L$-moves on labeled braidoid diagrams from \cite{GL2}: An $L$-move on a braidoid diagram is defined in the same way as in the classical case and the only difference lies in the fact that after applying an $L$-move, a new pair of corresponding strands appear, and in order for the labeled braidoid to remain labeled, we assign to the new pair of corresponding strands the label ``o'' or ``u'' according to the type of $L$-move that was applied.

\smallbreak

Similarly, we have the following:

\begin{defn}\rm
An {\it $L$-move} on a labeled pseudo braidoid $\beta$, consists in cutting an arc of $\beta$ open and pulling the upper cutpoint downward and the lower upward, so as to create a new pair of braid strands with corresponding endpoints (on the vertical line of the cutpoint), and such that both strands cross the rest of the braid only with real crossings (all over in the case of an $L_o$-move or all under in the case of an $L_u$-move) and we assign to the new pair of corresponding strands the label ``o'' or ``u'' according to the type of $L$-move that was applied. We shall call $L$-moves on labeled pseudo braidoids, {\it pseudo} $L$-{\it moves}. 
\end{defn}

In \cite{GL2}, the authors also define the {\it fake forbidden moves} on a labeled braidoid diagram $B$, as forbidden moves on $B$ which upon closure induce a sequence of fake forbidden moves on the resulting (multi-)knotoid diagram. Moreover, a {\it fake swing move} is defined as a swing move which is not restricted, in the sense that the endpoint surpasses the vertical line of a pair of corresponding ends but in the closure it gives rise to a sequence of swing and fake forbidden moves on the resulting (multi-)knotoid diagram. See Figure~\ref{fm2} for an example of a fake swing move and a fake forbidden move on a labeled braidoid diagram.

\begin{figure}[ht]
\begin{center}
\includegraphics[width=3.7in]{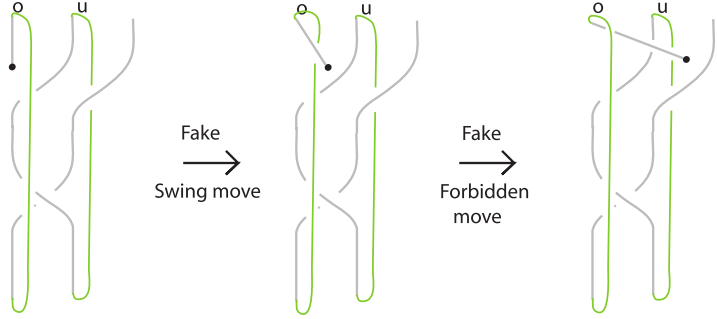}
\end{center}
\caption{A fake swing move and a fake forbidden move.}
\label{fm2}
\end{figure}

As proved in Lemma~8 \cite{GL2}, a fake forbidden move can be generated by a sequence of $L$-moves, together with planar isotopy and fake swing moves. Thus, they define $L${\it -equivalence} on labeled braidoid diagrams to be the equivalence relation on labeled braidoid diagrams generated by the $L$-moves together with labeled braidoid isotopy moves and fake swing moves. $L$-equivalence turns out to be especially useful for formulating a braidoid equivalence since there is no algebraic structure for braidoids. Indeed, we have the following result \cite{GL2}:

\begin{thm}[{\bf An analogue of the Markov theorem for braidoids}]\label{brisomark}
The closures of two labeled braidoid diagrams are isotopic (multi)-knotoids in $\Sigma$ if and only if the labeled braidoid diagrams are related to each other via $L$-equivalence moves.
\end{thm}

Following the same ideas, we also define the {\it fake pseudo forbidden moves} on a labeled pseudo braidoid diagram $B$, as pseudo forbidden moves on $B$ which upon closure induce a sequence of fake pseudo forbidden moves on the resulting pseudo (multi-)knotoid diagram. Moreover, a {\it fake pseudo swing move} is defined as a (pseudo) swing move which is not restricted, in the sense that the endpoint surpasses the vertical line of a pair of corresponding ends but in the closure it gives rise to a sequence of swing and fake pseudo forbidden moves on the resulting pseudo (multi-)knotoid diagram. See Figure~\ref{pfm2} for an example of a fake pseudo swing move and a fake pseudo forbidden move on a labeled pseudo braidoid diagram.

\begin{figure}[ht]
\begin{center}
\includegraphics[width=3.7in]{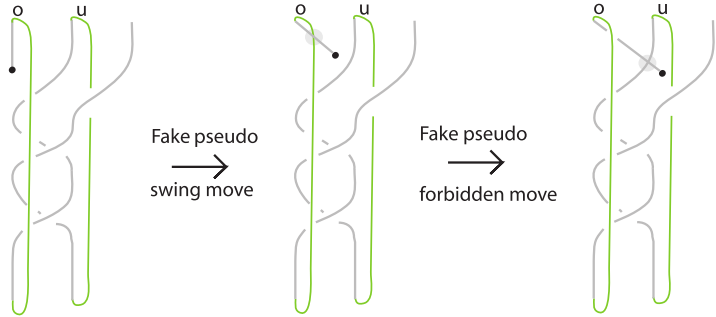}
\end{center}
\caption{A fake pseudo swing move and a fake pseudo forbidden move.}
\label{pfm2}
\end{figure}

We now define pseudo $L$-equivalence on pseudo braidoid diagrams as follows:

\begin{defn}\rm
The pseudo $L$-moves together with labeled pseudo braidoid isotopy moves and fake swing moves generate an equivalence relation on labeled pseudo braidoid diagrams that is called {\it pseudo} $L${\it -equivalence}.
\end{defn}

We are now in position to state and prove a theorem of pseudo braidoid equivalence:

\begin{thm}[{\bf An analogue of the Markov theorem for pseudo braidoids}]\label{breqpbroid}
The closures of two labeled pseudo braidoid diagrams are isotopic pseudo (multi)-knotoids in $\Sigma$ if and only if the labeled pseudo braidoid diagrams are related to each other via pseudo $L$-equivalence moves.
\end{thm}

\begin{proof}
The proof is similar to the proof of Thereom~3 in \cite{GL2} (see Theorem~\ref{brisomark} in this paper). We only illustrate here how a pseudo Reidemeister 1 move on a pseudo knotoid can be translated to pseudo $L$-equivalence and how a pseudo fake forbidden move on a (labeled) braidoid diagram is generated by (pseudo) $L$-moves, planar isotopy moves and pseudo swing moves. The rest of the cases are similar to those in \cite{GL2}.

\smallbreak

Figure~\ref{pr1markov} illustrates how a PR1 move is transformed to a pseudo $L$-move after ``braidoiding'' the initial pseudo knotoid, using a similar algorithm to that described in \S~\ref{alse} (for more details the reader is referred to \cite{GL2}). Note that the PR1 move in Figure~\ref{pr1markov} is assumed to take place away from the endpoints of the pseudo knotoid.

\begin{figure}[ht]
\begin{center}
\includegraphics[width=4in]{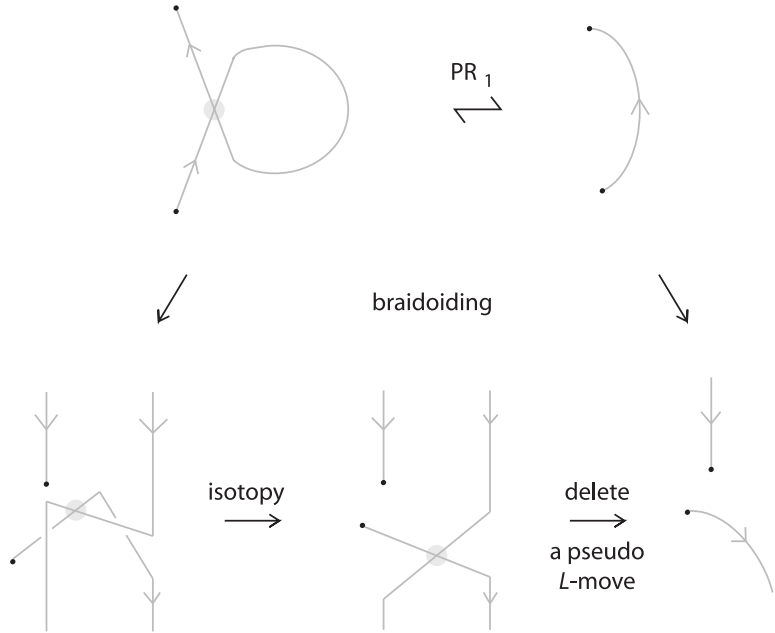}
\end{center}
\caption{A PR1 move on the level of pseudo braidoids.}
\label{pr1markov}
\end{figure}

Moreover, Figure~\ref{pffmmar} illustrates how a pseudo fake forbidden move on a (labeled) pseudo braidoid diagram is shown to be generated by (pseudo) $L$-moves and a swing move.

\begin{figure}[ht]
\begin{center}
\includegraphics[width=4in]{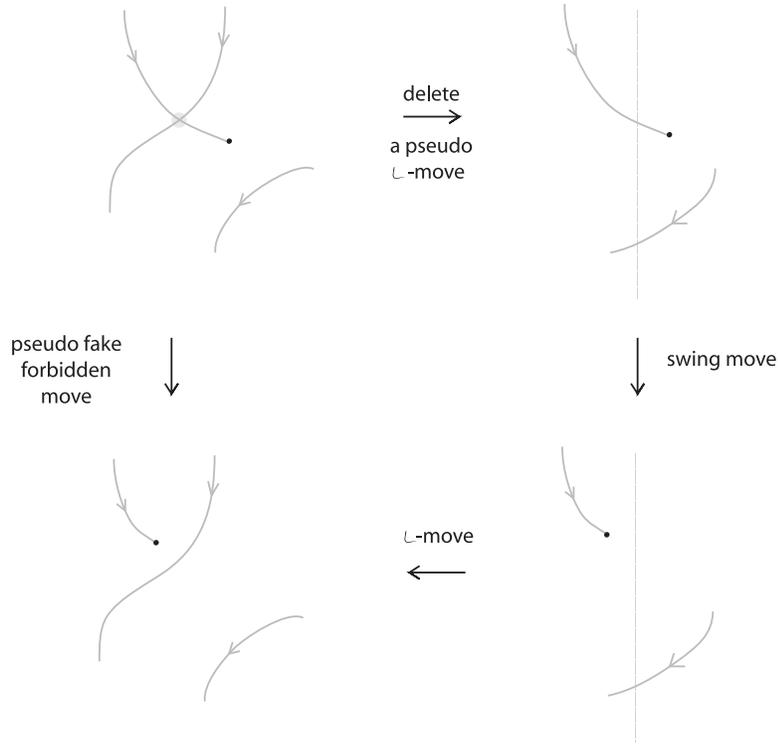}
\end{center}
\caption{A pseudo fake forbidden move on the level of pseudo braidoids.}
\label{pffmmar}
\end{figure}

\end{proof}

\section{Further Research}\label{fres}

Other families of interesting knotted objects that are worth studying, are the {\it tied multi-knotoids/linkoids} \& the {\it tied pseudo multi-knotoids/linkoids} defined as follows:

\begin{defn}\rm \label{furtherr}
A {\it tied multi-knotoid} is a multi-knotoid equipped with ties and a {\it tied linkoid} is a linkoid equipped with ties. Moreover, a {\it tied pseudo multi-knotoid/linkoid} is a pseudo multi-knotoid/linkoid equipped with ties. For an illustration see Figure~\ref{tpk}.
\end{defn}

\begin{figure}[ht]
\begin{center}
\includegraphics[width=2.2in]{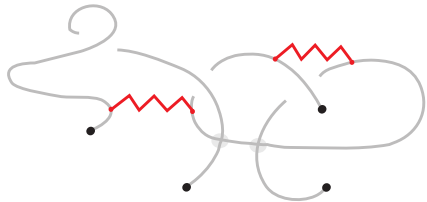}
\end{center}
\caption{A tied pseudo linkoid.}
\label{tpk}
\end{figure}

As in the classical case of tied links, the sets of ties in these knotted objects define a partition of the set of components. Thus, isotopy on the level of each knotted object introduced in Definition~\ref{furtherr} is translated by appropriate isotopy (ignoring the ties each time), and then taking into consideration that the set of ties in the knotted objects define the same partition of the set of components.

\smallbreak

In a sequel paper we shall study these knotted objects by defining the {\it tied braidoids} and the {\it tied pseudo braidoids} and work toward a braiding algorithm for these objects and also toward an algebraic statement of braidoid equivalence. Note that this is a challenging task even for the simple case of braidoids.

\section{Conclusions}

In this paper we introduce and study tied pseudo links via the tied pseudo braid monoid that is related to the tied singular braid monoid. Moreover, we introduce and study pseudo knotoids and pseudo braidoids. We believe that tied pseudo links and pseudo knotoids can be used to model biological objects related to DNA and our results set the bases for the construction of tied pseudo link invariants in the sense of \cite{Jo}. This method has been successfully applied in the case of tied links in $S^3$ in \cite{AJ4}, for tied links in the Solid Torus in \cite{F} and in \cite{D}, braiding algorithms and braiding equivalence for tied links in various 3-manifolds are presented and work toward the study of tied links in knot complements and other c.c.o. 3-manifolds is currently being done \cite{D4}. Finally, it is worth mentioning that this braid approach in defining invariants for knots in various 3-manifolds has been implemented for the construction of HOMFLYPT type invariants for knots in the Solid Torus in \cite{La1, DL2} and work toward the construction of such invariants for knots in lens spaces $L(p,1)$ has been done in \cite{DL3, DL4, DLP, D3, DL1}. For the case of Kauffman bracket type invariants via braids the reader is referred to \cite{D1} for knots in the Solid Torus and in \cite{D2} for knots in the genus 2 handlebody. In a sequel paper we shall adopt this braid approach for the study of other knot theories.

\end{document}